 \documentclass[final,3p,times]{elsarticle}
\usepackage{amssymb}
\usepackage{amsmath}
 \usepackage{amsthm}
\usepackage{kpfonts}

\journal{ }

\numberwithin{equation}{section}
\newtheorem{theorem}{Theorem}[section]
\newtheorem{lemma}[theorem]{Lemma}

\newtheorem{definition}[theorem]{Definition}

\newtheorem{remark}[theorem]{Remark}

\allowdisplaybreaks

\begin{document}

\begin{frontmatter}



\title{Analysis of the weak formulation of a coupled nonlinear parabolic system modeling a heat exchanger} %% Article title


\author[1]{Kouma Ali Ouattara} \ead{benzeus96@gmail.com}
\author[1]{Bomisso Gossrin Jean-Marc} \ead{bogojm@yahoo.fr}
\author[1]{B\'erenger Akon Kpata}\ead{kpata\_akon@yahoo.fr} 
\author[2]{Tour\'e Kidj\'egbo Augustin}\ead{latoureci@gmail.com}

\address[1]{Universit\'e Nangui Abrogoua, UFR Sciences Fondamentales et Appliqu\'ees, 02 BP 801 Abidjan 02, C\^ote d'Ivoire}
\address[2]{Institut National Polytechnique Houphou\"et-Boigny de Yamoussoukro, BP 1093 Yamoussoukro, C\^ote d'Ivoire}

%% Abstract
	\begin{abstract}
		This paper establishes the existence, uniqueness and time-space regularity of the weak solution to a nonlinear coupled parabolic system modeling temperature evolution in a coaxial heat exchanger with source terms and spatially varying coefficients. The system is formulated in a weak sense and the analysis relies on a Faedo-Galerkin method tailored to handle the nonlinear coupling and heterogeneous domains. Under suitable assumptions on the initial data and source terms, enhanced regularity in both time and space is obtained. In contrast with classical scalar models, the study addresses a multi-component system with realistic boundary conditions and complex interfacial dynamics.
	\end{abstract}

%% Keywords
\begin{keyword}
Heat exchanger \sep   parabolic system \sep weak formulation \sep existence \sep uniqueness \sep regularity result

\MSC[2020] 35K05 \sep 35K55 \sep 35A15 \sep 35A01 \sep 35A02 \sep 35A09

\end{keyword}

\end{frontmatter}

\section{Introduction}

The modeling of heat transfer in heterogeneous and structured media, characterized by strong thermal and mechanical couplings, has attracted growing attention from both theoretical and practical perspectives. In particular, this challenge arises in the design of high-performance energy systems including heat exchangers (see for example \cite{Bell2020}, \cite{Kakac2012} and \cite{Pawlowski2021}), which involve multiple domains with distinct thermodynamic and material properties, where heat transfer occurs through diffusion, convection and solid-to-solid heat transfer interactions.\\
\noindent In this work, we focus on a coaxial heat exchanger consisting of several concentric regions, each characterized by distinct thermophysical properties. The thermal response of the entire system is governed by the interactions between the different regions, which are described by the following system of nonlinear coupled parabolic partial differential equations $(\Sigma_T)$:
\begin{equation*}
	\partial_t \mathbf{T}(x,t) = \partial_x \left( \mathbb{E}(x)  \partial_x \mathbf{T}(x,t)  \right) - \mathbb{F}  \partial_x \mathbf{T}(x,t)  + \mathbb{K}(x) \mathbf{T}(x,t) + \mathbb{S}(t), \quad x \in (0,1), \; t >0.
\end{equation*}
Here, $ \mathbf{T} = \left( T_{\alpha}  \right)_{\alpha \in \mathcal{I}}^{\top} $ denotes the vector of unknown temperature functions, where each index $ \alpha \in \mathcal{I} = \{f, s, g, p\} $ refers respectively to the fluid ($ f $), the separating wall ($ s $), the gas ($ g $) and the insulating wall ($ p $).
 The matrices $ \mathbb{E} $, $ \mathbb{F} $, $ \mathbb{K} $ and the vector $ \mathbb{S}(t) $ represent the diffusion, convection, coupling (reaction) and source terms, respectively, and are defined as follows:
 \begin{itemize}
 	\item $	\mathbb{E}(x)=\mathrm{diag} \left( E_{f}(x), \; E_s(x), \; E_g(x), \; E_p(x) \right),$ where $ E_{\alpha} $ denotes the diffusion coefficient for the region \( \alpha \in \mathcal{I} \) and is assumed to be continuous, differentiable, strictly positive and bounded on the interval \( [0, 1] \).
 	\item  $ \mathbb{F} = \mathrm{diag} \left( f_f, 0, -f_g, 0 \right),$ where $ f_f $ and $ f_g $ are the convection coefficients for the fluid and gas regions, respectively, and satisfy $ f_f, \; f_g > 0 $, while the other regions have no convection.
 	\item $ \mathbb{K}(x) = \begin{pmatrix}
		-K_1(x)& K_1(x)&0&0\\
		K_2(x)& -[K_2+K_3](x)&K_3(x)&0\\
		0&K_4(x)& -[K_4+K_5](x)&K_5(x)\\
		0&0&K_6(x)&-K_6(x)
	\end{pmatrix},$ where $ K_j $ for $ j = 1, \dots, 6 $ are positive continuous functions on $ (0,1) $, representing the heat exchange rates between regions.
	\item $\mathbb{S}(t) = \left( S_{\alpha}(t) \right)_{\alpha \in \mathcal{I}}^{\top},$ where $ S_{\alpha} $ denotes the source term associated with each region, assumed to be Lipschitz continuous with respect to time, continuously differentiable and nonnegative.
	
\end{itemize}
The system $(\Sigma_T)$ is subject to the following boundary and initial conditions:
 \begin{itemize}
 	\item $T_f(0,t) = f_0, \quad T_g(1,t) = g_0, \quad t > 0$,
 where $ f_0 $ and $ g_0 $ are the given positive inlet temperatures for the fluid and the gas, respectively.
 	\item $	\left(E_{\alpha_1} (.)\, \partial_x T_{\alpha_1} (.,t)\right)(0) = \left(E_{\alpha_2} (.)\, \partial_x T_{\alpha_2} (.,t)\right)(1) = 0 , \quad t > 0, \; \alpha_1 \in \{g, s, p\}, \; \alpha_2 \in \{f, s, p\}.$
 	These conditions describe the absence of heat flux at the boundaries.
 	\item $	\mathbf{T}_0 = \mathbf{T}(x,0) = \left( T_{\alpha}(x,0) \right)_{\alpha \in \mathcal{I}}^{\top} = \left( T_{\alpha}^0(x) \right)_{\alpha \in \mathcal{I}}^{\top}, \quad x \in [0,1], \; \alpha \in \mathcal{I},$ where $ T_{\alpha}^0 \geq 0 $ represents the initial temperature distribution in each region $ \alpha \in \mathcal{I}$. 	
 \end{itemize}
 
\noindent The asymptotic behavior of the coupled system $(\Sigma_T)$ has been extensively studied under simplified assumptions: constant coefficients, linearized dynamics, neglected diffusion/source terms and homogeneous boundary conditions (see \cite{Xao}, \cite{Hideki}, \cite{Rauch} and \cite{Xu}). To gain further insight into the system's behavior, Abdoua et al. in \cite{Xao} proposed a numerical study of a linearized version of $(\Sigma_T)$, but this approach has not yet been applied. In order to conduct such an analysis, it is essential to first establish a rigorous theoretical foundation that ensures the existence and well-posedness of a weak solution. However, classical methods used to prove the existence of the weak solution (the inf-sup condition \cite{Brezzi1974}, fixed-point techniques \cite{Deimling1985}, energy minimization methods \cite{Lions1969}) are generally not employed in the implementation of concrete numerical approximations.\\
\indent In this work, our aim is to establish the existence, uniqueness and regularity of the weak solution to the general system $(\Sigma_T)$ using the Faedo-Galerkin method, which is traditionally applied to scalar equations (see \cite{L.C. Evans}, \cite{Hamrouni} and \cite{Yapi}). This method is particularly suited for numerical implementation, as it constructs approximate solutions in finite-dimensional spaces and naturally leads to schemes compatible with computational approaches.
 The main difficulty lies in the treatment of the coupling terms, which hinder the direct application of standard techniques developed for scalar or decoupled systems. In particular, the off-diagonal structure of the reaction matrix $\mathbb{K}$ leads to nontrivial cross-interactions that must be carefully controlled to obtain uniform energy estimates. To overcome these challenges, we adapt the Faedo-Galerkin method to the coupled setting, incorporating a recent basis-change technique introduced by Bomisso et al. in \cite{Kouma}. This strategy is used to handle the difficulties caused by coupling at the finite-dimensional approximation level and facilitating the derivation of key a priori estimates.\\
\noindent Our contribution consists in developing a comprehensive analytical framework for strongly coupled systems, extending classical existence and regularity results to a broader class of parabolic models. We address the challenges posed by cross-coupling interactions and derive uniform energy estimates that are essential for both theoretical insight and future numerical approximations.\\
\indent The paper is organized as follows. Section 2 deals with the notations used throughout the article. In Section 3, we give the definition of the weak solution adapted to the coupled structure of the system. Section 4 is devoted to the proof of existence and uniqueness of the weak solution of $(\Sigma_T)$, based on a Faedo-Galerkin method. We establish, in Section 5, an additional regularity result under further assumptions on the data.

\section{Preliminaries}

This section collects the notations, function spaces and norms used throughout this paper.\\

	We begin by introducing the following Hilbert space
\begin{equation*}
	\mathbb{H}_2 = \left( L^2(0,1) \right)^4,  
\end{equation*}
endowed with its natural inner product, whose corresponding norm is denoted by $\left\Vert . \right\Vert_{\mathbb{H}_2}$. More precisely, for any $\boldsymbol{\phi} = \left( \boldsymbol{\phi}_1,\; \boldsymbol{\phi}_2, \; \boldsymbol{\phi}_3, \; \boldsymbol{\phi}_4 \right)^{\top} \in \mathbb{H}_2 $, 
\begin{equation*}
	\left\Vert \boldsymbol{\phi} \right\Vert_{\mathbb{H}_2}^2 = \displaystyle\sum_{j=1}^4 \left\Vert \boldsymbol{\phi}_j \right\Vert_{2}^2 = \displaystyle\sum_{j=1}^4 \langle \boldsymbol{\phi}_j, \boldsymbol{\phi}_j \rangle,
\end{equation*}
where $\langle \cdot, \cdot \rangle$ denotes the canonical inner product on $L^2(0,1)$ and $\| \cdot \|_2$ its associated norm.\\

\noindent For each region $\alpha \in \mathcal{I}$, let us define the following Hilbert spaces:
\begin{equation*}
	\mathcal{V}_f = \left\{ u \in H^1(0,1) : u(0) = 0 \right\}, \quad
	\mathcal{V}_g = \left\{ w \in H^1(0,1) : w(1) = 0 \right\} \; \text{ and } \;
	\mathcal{V}_s = \mathcal{V}_p = H^1(0,1)
\end{equation*}
endowed respectively with the inner products
		\begin{align*}
		\langle u_1, u_2 \rangle_{\mathcal{V}_f}  & =   \left\langle E_f \partial_x u_1, \partial_x u_2 \right\rangle +  \langle K_1 u_1 , u_2 \rangle, \quad  u_1, \, u_2 \in \mathcal{V}_f , \\
		\langle v_1, v_2 \rangle_{\mathcal{V}_s}   & =   \left\langle E_s \partial_x v_1, \partial_x v_2 \right\rangle +  \langle (K_2 + K_3) v_1 , v_2 \rangle, \quad  v_1, \, v_2 \in \mathcal{V}_s ,\\
		\langle w_1, w_2 \rangle_{\mathcal{V}_g}  & =   \left\langle E_g \partial_x w_1, \partial_x w_2 \right\rangle + \langle (K_4 + K_5 )  w_1 , w_2 \rangle, \quad  w_1, \, w_2 \in \mathcal{V}_g ,\\
		\langle z_1, z_2 \rangle_{\mathcal{V}_p}  & =   \left\langle E_p \partial_x z_1, \partial_x z_2 \right\rangle + \langle  K_6 z_1 , z_2 \rangle, \quad  z_1, \, z_2 \in \mathcal{V}_p
	\end{align*}
with the corresponding norms induced by these inner products, respectively denoted by $\left\Vert \cdot \right\Vert_{\mathcal{V}_{\alpha}},$ for each $\alpha \in \mathcal{I}$.

\noindent We then define the global product space 
\begin{equation*}
	\mathcal{V} = \prod_{\alpha \in \mathcal{I}} \mathcal{V}_{\alpha}
\end{equation*}
and its dual is
\begin{equation*}
\mathcal V' = \prod_{\alpha \in \mathcal{I}} \mathcal{V}_{\alpha}^{\prime},
\end{equation*}
where each space $\mathcal{V}_\alpha^{\prime}$ is understood as the dual of the corresponding Hilbert space $\mathcal{V}_\alpha$, for $\alpha \in \mathcal{I}$. Accordingly, we denote by $
	\left\langle \cdot, \cdot \right\rangle_{\mathcal{V}_{\alpha}', \mathcal{V}_{\alpha}}$
the duality pairing between $ \mathcal{V}_{\alpha}'$ and $ \mathcal{V}_{\alpha}$, for each $\alpha \in \mathcal{I}$.\\

For all $\boldsymbol{\Phi}_1=(u_1,v_1,w_1,z_1)^\top,\;     \boldsymbol{\Psi}_2=(u_2,v_2,w_2,z_2)^\top\in\mathcal V$, we define the operators $\mathcal{L},\;\mathcal{A}$ and $\mathcal{K}$ on $\mathcal{V}^2$, with values in $\mathbb{R}^4$, by
\begin{equation*}
	\mathcal{L}(\boldsymbol{\Phi}_1, \boldsymbol{\Psi}_2) = 
	\begin{pmatrix}
		\langle u_1, u_2 \rangle \\
		\langle v_1, v_2 \rangle \\
		\langle w_1, w_2 \rangle \\
		\langle z_1, z_2 \rangle
	\end{pmatrix}, \quad
	\mathcal{A}(\boldsymbol{\Phi}_1, \boldsymbol{\Psi}_2) = 
	\begin{pmatrix}
		a_f(u_1, u_2) \\
		a_s(v_1, v_2) \\
		a_g(w_1, w_2) \\
		a_p(z_1, z_2)
	\end{pmatrix}, \quad \mathcal{K}(\boldsymbol{\Phi}_1, \boldsymbol{\Psi}_2) = 
	\begin{pmatrix}
	\left\langle K_1 v_1, u_2 \right\rangle \\
	\left\langle K_2 u_1 + K_3 w_1, v_2 \right\rangle \\
	\left\langle K_4 v_1 + K_5 z_1, w_2 \right\rangle \\
	\left\langle K_6 w_1, z_2 \right\rangle
	\end{pmatrix},
\end{equation*}
where $a_{\alpha}(\cdot,\cdot)$ are the bilinear forms defined on the corresponding subspaces $\mathcal{V}_{\alpha}$, for each $\alpha \in \mathcal{I}$, as follows:
\begin{align*}
	a_f(u_1, u_2) & = \left\langle E_f \partial_x u_1, \partial_x u_2 \right\rangle 
	+ f_f \left\langle \partial_x u_1 , u_2 \right\rangle 
	+ \left\langle K_1 u_1 , u_2 \right\rangle, \quad
	a_s(v_1, v_2) = \left\langle E_s \partial_x v_1, \partial_x v_2 \right\rangle 
	+ \left\langle (K_2 + K_3) v_1, v_2 \right\rangle, \\
	a_g(w_1, w_2) &= \left\langle E_g \partial_x w_1, \partial_x w_2 \right\rangle 
	- f_g \left\langle \partial_x w_1, w_2 \right\rangle 
	+ \left\langle (K_4 + K_5) w_1, w_2 \right\rangle, \quad
	a_p(z_1, z_2) = \left\langle E_p \partial_x z_1, \partial_x z_2 \right\rangle 
	+ \left\langle K_6 z_1, z_2 \right\rangle \cdotp
\end{align*}
Note that the operator $ \mathcal{K} $ encodes the coupling terms of the reaction matrix $ \mathbb{K} $, which are often referred to as the block-off-diagonal components.\\

We will also use the notation \( \mathcal{L}^*(\cdot, \cdot) \) to denote the same operator as \( \mathcal{L}(\cdot, \cdot) \), but where each \( L^2(0,1) \) inner product is replaced by the duality pairing between \( \mathcal{V}_{\alpha}' \) and \( \mathcal{V}_{\alpha} \).

We adopt the notation \( X \preceq Y \) for vectors \( X = (x_1, \dots, x_n)^{\top}, \; Y = (y_1, \dots, y_n)^{\top} \in \mathbb{R}^n \) to denote componentwise inequality, that is,
\[
X \preceq Y \quad \Longleftrightarrow \quad x_i \leq y_i, \quad \text{for all } i = 1, \dots, n.
\]

\noindent Let $T>0$ be fixed. Consider a vector-valued function \( \boldsymbol{u}(t) = \left(u_1(t), u_2(t), \dots, u_n(t) \right)^{\top} \) defined on \( [0,T] \), where each component \( u_j \), \( j = 1, \dots, n \), is integrable on \( [0,T] \). Then, we define the integral of \( \boldsymbol{u} \) over \( [0,T] \) componentwise by the notation
\[
\displaystyle\int_0^T \boldsymbol{u}(t) \, \text{d}t = \begin{pmatrix}
	\displaystyle\int_0^T u_1(t) \, \text{d}t \\
	\displaystyle\int_0^T u_2(t) \, \text{d}t \\
	\vdots \\
	\displaystyle\int_0^T u_n(t) \, \text{d}t
\end{pmatrix} \cdotp
\]
This notation will be used throughout the work when integrating vector-valued functions componentwise over a given interval.

\section{Definition of the weak solution}
The purpose of this section is to define the weak solution of the problem $(\Sigma_T)$.\\

Set $\mathbf{U} = \left( U_{\alpha} \right)_{\alpha \in \mathcal{I}}^{\top}$ such that
	\begin{equation*}
	\mathbf{U} = \mathbf{T} - \left( f_0,\; 0,\; g_0,\; 0 \right)^{\top}.
\end{equation*}

Thus, system $(\Sigma_T)$ is equivalent to the following:
	\begin{equation}
	\label{e1}
	\partial_t \mathbf{U}(x,t) = \partial_x \left( \mathbb{E}(x)  \partial_x \mathbf{U}(x,t)  \right) - \mathbb{F}  \partial_x \mathbf{U}(x,t)  + \mathbb{K}(x) \mathbf{U}(x,t) + \mathbb{S}(x,t), \quad x \in (0,1), \; t >0,
\end{equation}
with \begin{equation*}
	\mathbb{S}(x,t) = \left( S_i(x,t) \right)_{i=1,...,4}^{\top} = \begin{bmatrix}
		-f_0 K_1(x) + S_f(t)\\
		f_0 K_2(x) + g_0 K_3(x)  + S_s(t)\\
		- g_0 (K_4 + K_5)(x) + S_g(t)\\
		g_0 K_6(x) + S_p(t)
	\end{bmatrix},
\end{equation*}
subject to the following boundary and initial conditions:
	\begin{align}
	\label{e2}
	\begin{split}
		& U_f(0,t) = T_g(1,t) = 0, \quad t>0, \\
	   	& \left(E_{\alpha_1} (\cdot)\, \partial_x U_{\alpha_1} (.,t) \right)(0) = \left(E_{\alpha_2} (\cdot)\, \partial_x U_{\alpha_2} (.,t) \right)(1) = 0 , \quad t > 0, \; \alpha_1 \in \{g, s, p\}, \; \alpha_2 \in \{f, s, p\},\\
	   	& \mathbf{U}_0 = \mathbf{U}(x,0) = \left( U_{\alpha}(x,0) \right)_{\alpha \in \mathcal{I}}^{\top} = \left( U_{\alpha}^0(x) \right)_{\alpha \in \mathcal{I}}^{\top} \cdotp
	\end{split}
\end{align}

Let $ \boldsymbol{\phi} = \left(\boldsymbol{\phi}_1, \; \boldsymbol{\phi}_2, \; \boldsymbol{\phi}_3, \; \boldsymbol{\phi}_4   \right)^{\top} \in \mathcal{V}$. Multiply \eqref{e1} by $\text{ diag }\left( \boldsymbol{\phi}_1, \; \boldsymbol{\phi}_2, \; \boldsymbol{\phi}_3, \; \boldsymbol{\phi}_4  \right)$, integrate over $(0,1)$ and use the boundary conditions \eqref{e2} to obtain the following variational formulation:

 Find
\begin{equation*}
	 \mathbf{U} \,: \: (0,T] \longrightarrow \mathcal{V} \cap \left(H^2(0,1)\right)^4 
\end{equation*}
such that, for every $t > 0$ and for all $\boldsymbol{\phi} \in \mathcal{V}$, the following equalities hold:

\begin{equation}
	\label{e3}
		\mathcal{L} \left( \partial_t \mathbf{U}(t), \boldsymbol{\phi} \right) + \mathcal{A} \left(\mathbf{U}(t), \boldsymbol{\phi} \right) - \mathcal{K} \left(\mathbf{U}(t), \boldsymbol{\phi} \right) = \mathcal{L} \left(\mathbb{S}(x,t), \boldsymbol{\phi} \right)
\end{equation}
and
\begin{equation}
	\label{e4}
	\mathbf{U}_0 = \mathbf{U}(x,0) .
\end{equation}

\begin{remark}
	Equation \eqref{e1} implies that the time derivative 
	\(\partial_t \mathbf{U}\) belongs to the space $\mathcal{V}^{\prime}$.
\end{remark}

The above considerations lead to the following definition.

\begin{definition}
	\label{definition1}
	Let $ T > 0 $ be fixed. We say  a function 
	\begin{equation*}
		\label{e5}
		\mathbf{\hat{U}} = \left( \hat{U}_{\alpha} \right)_{\alpha \in \mathcal{I}}^{\top} \in L^2 \left( 0,T; \mathcal{V} \right) \text{ with } \partial_t \mathbf{\hat{U}} \in L^2 \left( 0,T; \mathcal{V}^{\prime} \right), 
	\end{equation*}
	is a weak solution to the problem \eqref{e3}-\eqref{e4} if
	\begin{itemize}
		\item[\textbf{(i)}] 
		\begin{equation*}
			\label{e6}
		\mathcal{L}^{\ast} \left( \partial_t \mathbf{\hat{U}}(t), \boldsymbol{\phi} \right) + \mathcal{A} \left(\mathbf{\hat{U}}(t), \boldsymbol{\phi} \right) - \mathcal{K} \left(\mathbf{\hat{U}}(t), \boldsymbol{\phi} \right) = \mathcal{L} \left(\mathbb{S}(t), \boldsymbol{\phi} \right)
		\end{equation*}
		for almost every \( t \in [0,T] \), for all \( \boldsymbol{\phi} \in \mathcal{V} \) and
		\item[\textbf{(ii)}]
		\begin{equation*}
			\label{e7}
			\mathbf{\hat{U}} (0) = \mathbf{\hat{U}}_0 = \mathbf{U}_0 \cdotp
		\end{equation*} 
		
	\end{itemize}

\end{definition}

\section{Existence and uniqueness of the weak solution}	
We aim to show that problem \eqref{e1}-\eqref{e2} admits a unique weak solution, depending continuously on the data in a suitable norm.
\subsection{Existence of the weak solution}
 We employ the Faedo-Galerkin method \cite{Lions1969}, which is also well-suited for numerical approximations.\\
We construct approximate solutions through a finite-dimensional Galerkin scheme, derive uniform a priori estimates and pass to the limit to obtain a weak solution, recover the regularity of $\partial_t \mathbf{\hat{U}}$ and verify the initial condition.

\subsubsection{Construction of approximate solutions}
We first construct approximate solutions using a finite-dimensional Galerkin scheme.\\
We state a lemma concerning the existence of a basis for the space $\mathcal{V}$.

\begin{lemma}
	\label{lemme1}
	There exists a sequence of functions $\left\{ \boldsymbol{\varPsi}_i^{(\alpha)} \right\}_{i=1}
^{\infty}$ for each $\alpha \in \mathcal{I}$, such that:
\begin{itemize}
	\item  $\left\{ e_{\alpha} \left( \boldsymbol{\varPsi}_i^{(\alpha)} \right): \; \alpha \in \mathcal{I} \right\}_{i=1}
	^{\infty}$ forms an orthogonal basis of $\mathcal{V}$,
	\item  $\left\{ e_{\alpha} \left( \boldsymbol{\varPsi}_i^{(\alpha)} \right): \; \alpha \in \mathcal{I} \right\}_{i=1}
	^{\infty}$ forms an orthonormal basis of $\mathbb{H}_2$,
\end{itemize}
where 
\begin{align*}
e_{\alpha} \left(\boldsymbol{\varPsi}_i^{(\alpha)}   \right) = \bigl(\delta_{\alpha f}\, \boldsymbol{\varPsi}_i^{(\alpha)},\;\delta_{\alpha s}\, \boldsymbol{\varPsi}_i^{(\alpha)},\;\delta_{\alpha g}\,\boldsymbol{\varPsi}_i^{(\alpha)},\;\delta_{\alpha p}\,\boldsymbol{\varPsi}_i^{(\alpha)}\bigr), \quad i \geq 1, \; \alpha \in \mathcal{I},
\end{align*}
with $\delta_{\alpha k}$ the Kronecker symbol on $\mathcal{I}$.
\end{lemma}
\begin{proof}
	 Since each space $ \mathcal{V}_\alpha $, with $\alpha \in \mathcal{I}$, is separable, each admits an orthogonal basis $ \left\{\boldsymbol{\varPsi}_i^{(\alpha)}\right\}_{i=1}^{\infty} $(see \cite{Daniel}, Theorem 2.4.3).  
		As a consequence, the space $\mathcal{V}$ also admits an orthogonal basis  $\left\{ e_{\alpha} \left( \boldsymbol{\varPsi}_i^{(\alpha)} \right): \; \alpha \in \mathcal{I} \right\}_{i=1}
		^{\infty}$.\\
		
		\noindent For the second assertion, observe that $\mathcal{V}$ is an infinite-dimensional separable Hilbert space and thus is isomorphic to $\mathbb{H}_2$ (see \cite{Daniel}, Theorem 2.3.15 and Corollary 2.4.4). Hence, $\left\{ e_{\alpha} \left( \boldsymbol{\varPsi}_i^{(\alpha)} \right): \; \alpha \in \mathcal{I} \right\}_{i=1}
		^{\infty}$
		is an orthonormal basis of $\mathbb{H}_2$.
\end{proof}

\begin{theorem}[\textbf{Construction of approximate solutions}]
	\label{theoreme1}
Let $T>0$ be fixed.	For all $m \in \mathbb{N}^{\ast}$, there exists a unique function $\mathbf{\hat{U}}^m = \left( \hat{U}_{\alpha}^m \right)_{\alpha \in \mathcal{I}}^{\top}$ of the form
	\begin{equation}
		\label{e8}
		\mathbf{\hat{U}}^m (x,t)  = \left( \displaystyle\sum_{j=1}^m \alpha_j^m (t) \, \boldsymbol{\varPsi}_j^{(\alpha)} (x) \right)_{\alpha \in \mathcal{I}}^{\top},\quad \alpha_j^m \in C \left([0,T], \mathbb{R} \right) \; (j=1,...,m, \; \alpha \in \mathcal{I})
	\end{equation}
	satisfying for all $i=1,...,m,$
	\begin{equation}
		\label{e9}
		\mathcal{L} \left( \partial_t \mathbf{\hat{U}}^m(t),  \left( \boldsymbol{\varPsi}_i^{(\alpha)} \right)_{\alpha \in \mathcal{I}}^{\top} \right) + \mathcal{A} \left(\mathbf{\hat{U}}^m(t), \left( \boldsymbol{\varPsi}_i^{(\alpha)} \right)_{\alpha \in \mathcal{I}}^{\top} \right) - \mathcal{K} \left(\mathbf{\hat{U}}^m(t), \left( \boldsymbol{\varPsi}_i^{(\alpha)} \right)_{\alpha \in \mathcal{I}}^{\top} \right) = \mathcal{L} \left(\mathbb{S}(x,t), \left( \boldsymbol{\varPsi}_i^{(\alpha)} \right)_{\alpha \in \mathcal{I}}^{\top} \right)
	\end{equation}
	for almost every $ t \in [0,T] $ and
	\begin{equation}
		\label{e10}
		\left( \alpha_j^m(0) \right)_{\alpha \in \mathcal{I}}^{\top} = \mathcal{L} \left( \mathbf{\hat{U}}_0, \left( \boldsymbol{\varPsi}_j^{(\alpha)} \right)_{\alpha \in \mathcal{I}}^{\top} \right), \quad j=1,...,m . 
	\end{equation}
\end{theorem}

\begin{proof}
	Let $T>0$ and $m \in \mathbb{N}^{\ast}$ be fixed. Assuming $\mathbf{\hat{U}}^m$ has the structure \eqref{e8}, we first note from Lemma \ref{lemme1} that, for each fixed $i \in \{ 1,...,m \},$
	\begin{align*}
	\mathcal{L} \left( \partial_t \mathbf{\hat{U}}^m(t), \left( \boldsymbol{\varPsi}_i^{(\alpha)} \right)_{\alpha \in \mathcal{I}}^{\top} \right) = \left( \dfrac{\text{d}}{\text{d}t} \alpha_i^m (t) \right)_{\alpha \in \mathcal{I}}^{\top}.
	\end{align*}
	Futhermore, for fixed $i \in \{ 1,...,m \},$
	\begin{align*}
		\mathcal{A} \left(\mathbf{\hat{U}}^m(t), \left( \boldsymbol{\varPsi}_i^{(\alpha)} \right)_{\alpha \in \mathcal{I}}^{\top} \right) = \left( \displaystyle\sum_{j =1}^m \beta_{\alpha}^{i j} \alpha_j^m(t)  \right)_{\alpha \in \mathcal{I}}^{\top}, 
	\end{align*}
	with $\beta_{\alpha}^{i j} = a_{\alpha} \left( \boldsymbol{\varPsi}_{j}^{(\alpha)}, \boldsymbol{\varPsi}_i^{(\alpha)}\right)$, $j \in \left\{ 1,...,m \right\} $ and for $ \alpha \in \mathcal{I} .  $\\
	\noindent Next, the presence of the coupling operator $ \mathcal{K} $, involving functions from distinct approximation spaces, prevents the direct use of standard Galerkin projections. Following Bomisso et al. \cite{Kouma}, we perform a change of basis to express each function $ \boldsymbol{\varPsi}_j^{(\alpha)}, \; \alpha \in \mathcal{I}, $ in the basis associated with the equation where it is tested. Thus, for each $\alpha, \eta \in \mathcal{I}$ and each $j=1,...,m$, there exist real coefficients $\gamma_j^{k,(\alpha,\eta)}$ such that
	\begin{equation*}
		\label{change_basis}
	\boldsymbol{\varPsi}_j^{(\alpha)} = \sum_{k=1}^m \gamma_j^{k,(\alpha,\eta)} \, \boldsymbol{\varPsi}_k^{(\eta)}.
	\end{equation*}
	It follows for all $i = 1,...,m,$ that
	\begin{equation*}
		 \mathcal{K} \left(\mathbf{\hat{U}}^m(t), \left( \boldsymbol{\varPsi}_i^{(\alpha)} \right)_{\alpha \in \mathcal{I}}^{\top} \right) = \begin{bmatrix}
		 &	 \langle \displaystyle\sum_{j=1}^m s_j^m(t) K_1 \sum_{k=1}^m \gamma_j^{k,(s,f)} \, \boldsymbol{\varPsi}_k^{(f)} , \boldsymbol{\varPsi}_i^{(f)} \rangle\\
		  &	 \langle \displaystyle\sum_{j=1}^m \left[ f_j^m(t) K_2  \sum_{k=1}^m \gamma_j^{k,(f,s)} \, \boldsymbol{\varPsi}_k^{(s)}  +  g_j^m(t) K_3  \sum_{k=1}^m \gamma_j^{k,(g,s)} \, \boldsymbol{\varPsi}_k^{(s)} \right] , \boldsymbol{\varPsi}_i^{(s)} \rangle\\
		  &	 \langle \displaystyle\sum_{j=1}^m \left[ s_j^m(t) K_4  \sum_{k=1}^m \gamma_j^{k,(s,g)} \, \boldsymbol{\varPsi}_k^{(g)}  +  p_j^m(t) K_5  \sum_{k=1}^m \gamma_j^{k,(p,g)} \, \boldsymbol{\varPsi}_k^{(g)} \right] , \boldsymbol{\varPsi}_i^{(g)} \rangle\\
		  &	 \langle \displaystyle\sum_{j=1}^m g_j^m(t) K_6 \sum_{k=1}^m \gamma_j^{k,(g,p)} \, \boldsymbol{\varPsi}_k^{(p)} , \boldsymbol{\varPsi}_i^{(p)} \rangle
		 \end{bmatrix} .
	\end{equation*}
	  
	\noindent Consequently, \eqref{e9} becomes a system of $4m$ ordinary differential equations with $4m$ unknowns, taking the form
	 \begin{equation}
	 	\label{e11}
	 	 \dfrac{\text{d}}{\text{d}t} \boldsymbol{\alpha}^m(t) + A \boldsymbol{\alpha}^m(t) = S(t),
	 \end{equation}
	 where $\boldsymbol\alpha^m(t)=\bigl( f^m(t), s^m(t), g^m(t), p^m(t)\bigr)^\top\in\mathbb R^{4m}$. The matrix $A$ is structured in $ 4\times4 $ blocks of size $m$:
	 \begin{equation*}
	 A
	 =
	 \begin{pmatrix}
	 	A_{ff} & A_{fs} & 0      & 0      \\
	 	A_{sf} & A_{ss} & A_{sg} & 0      \\
	 	0      & A_{gs} & A_{gg} & A_{gp} \\
	 	0      & 0      & A_{pg} & A_{pp}
	 \end{pmatrix}.
	 \end{equation*}
	 Here the diagonal blocks are
	 \begin{equation*}
	 \bigl(A_{\alpha\alpha}\bigr)_{ij}
	 \;=\;\beta_{\alpha}^{i j},
	 \quad i,j=1,...,m, \; \alpha \in \mathcal{I}
	 \end{equation*}
	 and the off-diagonal coupling blocks for $\alpha \neq \beta$, encoding the interactions between the different components of the system, are given, for $i,j=1,...,m$, by
	 \begin{align*}
	 \begin{aligned}
	 	(A_{fs})_{ij}
	 	&=\langle K_1\sum_{k=1}^m\gamma_j^{k,(s,f)}\,\boldsymbol\varPsi_k^{(f)},\,
	 	\boldsymbol\varPsi_i^{(f)} \rangle, \;
	 	(A_{sf})_{ij}
	 	=\langle K_2\sum_{k=1}^m\gamma_j^{k,(f,s)}\,\boldsymbol\varPsi_k^{(s)},\,
	 	\boldsymbol\varPsi_i^{(s)} \rangle, \;
	 	(A_{sg})_{ij}
	 	=\langle K_3\sum_{k=1}^m\gamma_j^{k,(g,s)}\,\boldsymbol\varPsi_k^{(s)},\,
	 	\boldsymbol\varPsi_i^{(s)} \rangle, \\[6pt]
	 	(A_{gs})_{ij}
	 	&= \langle K_4\sum_{k=1}^m\gamma_j^{k,(s,g)}\,\boldsymbol\varPsi_k^{(g)},\,
	 	\boldsymbol\varPsi_i^{(g)} \rangle, \;
	 	(A_{gp})_{ij}
	 	=\langle K_5\sum_{k=1}^m\gamma_j^{k,(p,g)}\,\boldsymbol\varPsi_k^{(g)},\,
	 	\boldsymbol\varPsi_i^{(g)} \rangle, \;
	 	(A_{pg})_{ij}
	 	=\langle K_6\sum_{k=1}^m\gamma_j^{k,(g,p)}\,\boldsymbol\varPsi_k^{(p)},\,
	 	\boldsymbol\varPsi_i^{(p)} \rangle \cdotp
	 \end{aligned}
	 \end{align*}
	 The each source-vector block \(S_\alpha(t)\in\mathbb{R}^m\) has components
	 \begin{equation*}
	 \bigl(S_\alpha(t)\bigr)_i
	 =\mathcal{L} \left( \mathbb{S}(x,t), \left( \boldsymbol{\varPsi}_i^{(\alpha)} \right)_{\alpha \in \mathcal{I}}^{\top} \right),
	 \quad i=1,\dots,m.
	 \end{equation*}
	The system \eqref{e9} is subject to the initial conditions \eqref{e10}. By the Picard-Lindel\"of Theorem (see \cite{peano}), there exists a unique absolutely continuous function $ \boldsymbol{\alpha}^m(t) $ that satisfies both \eqref{e9} and \eqref{e11} for almost every $ t \in [0, T] $. Therefore, the function $ \mathbf{\hat{U}}^m$ defined by \eqref{e8} also solves \eqref{e9} for almost every $ t \in [0, T] $.
\end{proof}

\subsubsection{A priori estimates}
Our goal is to prove that the Galerkin approximation \eqref{e8} of $\mathbf{\hat{U}}^m$ satisfies uniform a priori estimates, allowing the extraction of convergent subsequences to a weak solution of the original problem. These results ensure the limit passage in the approximate system \eqref{e11} as $m \to \infty$.

\begin{theorem}
	\label{theoreme2}
Let $ T > 0 $ be fixed. Then, for every $ m \in \mathbb{N}^\ast $, there exists a constant $ C > 0 $ such that
	\begin{align*}
		\label{e12}
		\begin{split}
		\max_{0 \leq t \leq T} \left( \left\Vert \mathbf{\hat{U}}^m(t) \right\Vert_{\mathbb{H}_2}^2 + \left\Vert \mathbf{\hat{U}}^m \right\Vert_{\left(L^2\left(0,T; \mathcal{V}\right)\right)^4}^2 + \left\Vert \partial_t \mathbf{\hat{U}}^m \right\Vert_{\left(L^2\left(0,T; \mathcal{V}^{\prime}\right)\right)^4}^2 \right) \leq C \left( \left\Vert \mathbf{\hat{U}}_0 \right\Vert_{\mathbb{H}_2}^2 + \left\Vert \mathbb{S} \right\Vert_{\left(L^2\left(0,T; \mathbb{H}_2\right)\right)^4}^2 \right) \cdotp
		\end{split}
	\end{align*}
\end{theorem}
\begin{proof}
Let $T>0$ and $m \in \mathbb{N}^{\ast}$ be fixed.\\
\noindent \textbf{1.} We begin by estimating $ \left\Vert \mathbf{\hat{U}}^m(t) \right\Vert_{\mathbb{H}_2} $ in order to deduce a bound for $ \left\Vert \mathbf{\hat{U}}^m \right\Vert_{\left(L^2\left(0,T; \mathcal{V}\right)\right)^4} $. For $j=1,...,m,$ we multiply \eqref{e9} by $ \text{diag } \left(f_j^m(t),\; s_j^m(t),\; g_j^m(t),\; p_j^m(t) \right) $. Then, summing each resulting equation over \( j = 1, \dots, m \) and using the form \eqref{e8} of \( \mathbf{\hat{U}}^m \), we obtain the following identity for almost every $ t \in [0, T] $
\begin{equation*}
	\label{e13}
	\mathcal{L} \left( \partial_t \mathbf{\hat{U}}^m(t),  \mathbf{\hat{U}}^m(t) \right) + \mathcal{A} \left(\mathbf{\hat{U}}^m(t), \mathbf{\hat{U}}^m(t) \right) - \mathcal{K} \left(\mathbf{\hat{U}}^m(t), \mathbf{\hat{U}}^m(t) \right) = \mathcal{L} \left(\mathbb{S}(x,t), \mathbf{\hat{U}}^m(t) \right).
\end{equation*}
By the $1$-coercivity of the bilinear forms $ a_s $, $ a_g $ and by Theorem 2 in \cite{L.C. Evans} (pp. 300), there exist constants $ \beta_f, \beta_g > 0 $ and $ \gamma_f, \gamma_g \geq 0 $ such that
\begin{equation}
	\label{e14}
\text{ diag } \left( \beta_f, \; 1, \; \beta_g, \; 1 \right) \; \left(  \left\Vert \hat{U}_{\alpha}^m \right\Vert_{\mathcal{V}_{\alpha}}^2  \right)_{\alpha \in \mathcal{I}}^{\top} \preceq \mathcal{A} \left(\mathbf{\hat{U}}^m(t), \mathbf{\hat{U}}^m(t) \right) + \text{ diag } \left( \gamma_f, \; 0, \; \gamma_g, \; 0 \right) \; \left(  \left\Vert \hat{U}_{\alpha}^m \right\Vert_{\mathcal{V}_{\alpha}}^2  \right)_{\alpha \in \mathcal{I}}^{\top}.
\end{equation}
Applying the Cauchy-Schwarz and Young inequalities $ \left(ab \leq \dfrac{1}{2}(a^2 + b^2) \right) $ to each component of  vectors $ \mathcal{K} \left(\mathbf{\hat{U}}^m(t), \mathbf{\hat{U}}^m(t) \right)$ and $ \mathcal{L} \left(\mathbb{S}(x,t), \mathbf{\hat{U}}^m(t) \right) $, we obtain
\begin{equation}
	\label{e16}
\mathcal{K} \left(\mathbf{\hat{U}}^m(t), \mathbf{\hat{U}}^m(t) \right) \preceq M_1 \left(  \left\Vert \hat{U}_{\alpha}^m(t) \right\Vert_{2}^2  \right)_{\alpha \in \mathcal{I}}^{\top}  \quad 	\text{ and } \quad \mathcal{L} \left(\mathbb{S}(x,t), \mathbf{\hat{U}}^m(t) \right)  \preceq M_2 \left(  \left\Vert \hat{U}_{\alpha}^m(t) \right\Vert_{2}^2  \right)_{\alpha \in \mathcal{I}}^{\top} + M_3 \left(  \left\Vert \mathbb{S}_i(x,t) \right\Vert_{2}^2  \right)_{i=1,...,4}^{\top}
\end{equation}
for certain constants $M_1, \, M_2, \, M_3 >0$.\\
Also, we have for almost every $ t \in [0, T] $
\begin{equation}
	\label{e17}
	\mathcal{L} \left( \partial_t \mathbf{\hat{U}}^m(t),  \mathbf{\hat{U}}^m(t) \right) = \left( \dfrac{1}{2} \dfrac{\text{d}}{\text{d}t} \left\Vert U_{\alpha}^m (t) \right\Vert_2^2  \right)_{\alpha \in \mathcal{I}}^{\top} \cdotp
\end{equation}
By adding the inequalities in \eqref{e14}-\eqref{e16} componentwise and summing over the four subsystems, then combining with the identity in \eqref{e17}, we obtain the following estimate for almost every $ t \in [0, T] $,
\begin{equation}
	\label{e18}
	\dfrac{\text{d}}{\text{d}t} \left\Vert \mathbf{\hat{U}}^m(t) \right\Vert_{\mathbb{H}_2}^2 + \left\Vert \mathbf{\hat{U}}^m(t) \right\Vert_{\mathcal{V}}^2 \leq C_1 \left\Vert \mathbf{\hat{U}}^m(t) \right\Vert_{\mathbb{H}_2}^2 + C_2 \left\Vert \mathbb{S}(x,t) \right\Vert_{\mathbb{H}_2}^2
\end{equation}
for appropriate constants $C_1$ and $C_2$. Thus, the differential form of Gronwall's inequality (see \cite{L.C. Evans}) yields the estimate
\begin{equation*}
	\label{e19}
	\left\Vert \mathbf{\hat{U}}^m(t) \right\Vert_{\mathbb{H}_2}^2 \leq \mathrm{e}^{C_1 t} \left( \left\Vert \mathbf{\hat{U}}^m(0) \right\Vert_{\mathbb{H}_2}^2 + C_2 \displaystyle\int_0^t \left\Vert \mathbb{S}(x,s) \right\Vert_{\mathbb{H}_2}^2 \text{d}s \right), \quad 0 \leq t \leq T.
\end{equation*}
Using \eqref{e10}, which ensures that $ \left\Vert \mathbf{\hat{U}}^m(0) \right\Vert_{\mathbb{H}_2}^2 \leq \left\Vert \mathbf{\hat{U}}_0 \right\Vert_{\mathbb{H}_2}^2 $, we deduce the following estimate
\begin{equation*}
	\label{e20}
	\max_{0 \leq t \leq T} \left\Vert \mathbf{\hat{U}}^m(t) \right\Vert_{\mathbb{H}_2}^2 \leq  C \left(\left\Vert \mathbf{\hat{U}}_0 \right\Vert_{\mathbb{H}_2}^2 + \left\Vert \mathbb{S} \right\Vert_{L^2(0,T;\mathbb{H}_2)}^2 \right).
\end{equation*}
Resuming from inequality \eqref{e18}, we integrate both sides from $ 0 $ to $ T $ and use \eqref{e20} to derive
\begin{equation*}
	\label{e21}
	\left\Vert \mathbf{\hat{U}}^m \right\Vert_{\left(L^2(0,T; \mathcal{V})\right)^4}^2 \leq  C \left(\left\Vert \mathbf{\hat{U}}_0 \right\Vert_{\mathbb{H}_2}^2 + \left\Vert \mathbb{S} \right\Vert_{L^2(0,T;\mathbb{H}_2)}^2 \right).
\end{equation*}
\textbf{2.} It remains to estimate $\left\Vert \partial_t \mathbf{\hat{U}}^m \right\Vert_{\left(L^2(0,T; \mathcal{V}^{\prime})\right)^4}\cdotp$ Let \( \boldsymbol{\phi} \in \mathcal{V} \) be arbitrary with \( \left\Vert \boldsymbol{\phi} \right\Vert_{\mathcal{V}} \leq 1 \) and decompose it as \( \boldsymbol{\phi} = \boldsymbol{\phi}_1 + \boldsymbol{\phi}_2 \), where \( \boldsymbol{\phi}_1 \in \text{span} \left\{ e_{\alpha} \left( \boldsymbol{\varPsi}_i^{(\alpha)} \right) : \alpha \in \mathcal{I} \right\}_{i=1}^{\infty} \) and \( \boldsymbol{\phi}_2 \) lies to the orthogonal of this span in \( \mathcal{V} \).
 From Lemma \ref{lemme1}, we have 
\begin{equation*}
	\left\Vert \boldsymbol{\phi}_1 \right\Vert_{\mathcal{V}}^2 \leq  \left\Vert  \boldsymbol{\phi} \right\Vert_{\mathcal{V}}^2 \leq 1.
\end{equation*}
Using \eqref{e9}, we deduce for almost every $t \in [0,T]$ that
\begin{equation*}
	\mathcal{L} \left( \partial_t \mathbf{\hat{U}}^m(t),  \boldsymbol{\phi}_1 \right) + \mathcal{A} \left(\mathbf{\hat{U}}^m(t), \boldsymbol{\phi}_1 \right) - \mathcal{K} \left(\mathbf{\hat{U}}^m(t), \boldsymbol{\phi}_1 \right) = \mathcal{L} \left(\mathbb{S}(x,t), \boldsymbol{\phi}_1 \right).
\end{equation*}
Then, \eqref{e8} implies 
\begin{equation*}
	\label{e22}
	\mathcal{L}^{\ast} \left( \partial_t \mathbf{\hat{U}}^m, \boldsymbol{\phi} \right) = 	\mathcal{L} \left( \partial_t \mathbf{\hat{U}}^m, \boldsymbol{\phi} \right) = \mathcal{L} \left( \partial_t \mathbf{\hat{U}}^m, \boldsymbol{\phi}_1 \right) = \mathcal{L} \left(\mathbb{S}(x,t), \boldsymbol{\phi}_1 \right) - \mathcal{A} \left(\mathbf{\hat{U}}^m(t), \boldsymbol{\phi}_1 \right) + \mathcal{K} \left(\mathbf{\hat{U}}^m(t), \boldsymbol{\phi}_1 \right).
\end{equation*}
Consequently,
\begin{equation*}
	\left\Vert \partial_t \mathbf{\hat{U}}^m \right\Vert_{\mathcal{V}^{\prime}} \leq C \left( \left\Vert \mathbb{S}(x,t) \right\Vert_{\mathbb{H}_2} + \left\Vert \mathbf{\hat{U}}^m \right\Vert_{\mathcal{V}} \right)
\end{equation*}
and by integrating over $(0,T)$ we get
\begin{equation*}
	\label{e23}
	\left\Vert \partial_t \mathbf{\hat{U}}^m \right\Vert_{\left(L^2\left(0,T; \mathcal{V}^{\prime}\right)\right)^4}^2 \leq C \left(\left\Vert \mathbf{\hat{U}}_0 \right\Vert_{\mathbb{H}_2}^2 + \left\Vert \mathbb{S} \right\Vert_{L^2(0,T;\mathbb{H}_2)}^2 \right).
\end{equation*}
\end{proof}

\subsubsection{Passage to limits}
We proceed by passing to the limit as \( m \to \infty \) to obtain a weak solution to \eqref{e1}-\eqref{e2}.

\begin{theorem}[\textbf{Existence of weak solution}]
	\label{theoreme3}
	There exists a weak solution to \eqref{e1}-\eqref{e2}.
\end{theorem}
\begin{proof}
	Let $T>0$ be fixed. Theorem \eqref{theoreme2} shows that the sequence of Galerkin's approximations $\left\{ \mathbf{\hat{U}}^m \right\}_{m \in \mathbb{N}^{\ast}}$ of the form \eqref{e8} is bounded in $L^{\infty} \left(0,T; \mathcal{V} \right) $, hence in $L^{2} \left(0,T; \mathcal{V} \right) $, while $\left\{ \partial _t \mathbf{\hat{U}}^m \right\}_{m \in \mathbb{N}^{\ast}}$ is bounded in $L^{2} \left(0,T; \mathcal{V}^{\prime} \right) $.\\
	We now use the Eberlein-\v{S}mulian Theorem (see \cite{Daniel}) and deduce that there exists a subsequence  $\left\{ \mathbf{\hat{U}}^{m_{\ell}} \right\}_{\ell \in \mathbb{N}^{\ast}}$ and a function $\mathbf{\hat{U}} \in L^{2} \left(0,T; \mathcal{V} \right) $, with $ \partial_t \mathbf{\hat{U}} \in L^{2} \left(0,T; \mathcal{V}^{\prime} \right) $, such that, as $ \ell \to \infty$,
	\begin{align}
		\label{e24}
		\begin{split}
			\left \{
		\begin{aligned}
			\mathbf{\hat{U}}^{m_{\ell}} & \rightharpoonup \mathbf{\hat{U}} \quad \text{ weakly in } L^{2} \left(0,T; \mathcal{V} \right),\\
			\partial_t \mathbf{\hat{U}}^{m_{\ell}} & \rightharpoonup \partial_t \mathbf{\hat{U}} \quad \text{ weakly in } L^{2} \left(0,T; \mathcal{V}^{\prime} \right).
		\end{aligned}
		\right.
		\end{split}
	\end{align}
	Next, fix an integer $N$ and choose $\mathbf{V} \in L^{2} \left(0,T; \mathcal{V} \right)$ having the form
	\begin{equation}
		\label{e25}
		\mathbf{V} (t) = \left( \displaystyle\sum_{j=1}^N \xi_{j}^{(\alpha)}(t) \boldsymbol{\varPsi}_j^{(\alpha)} \right)_{\alpha \in \mathcal{I}}^{\top},
	\end{equation}
	where each $\left\{ \xi_{j}^{(\alpha)} \right\}_{j=1}^N$ for $\alpha \in \mathcal{I}$ are given smooth functions. Let $m \geq N$ and $T>0$ be fixed. For $j=1,...,m,$ we multiply \eqref{e9} by $ \text{diag } \left(\xi_{j}^{(f)}(t),\; \xi_{j}^{(s)}(t),\; \xi_{j}^{(g)}(t),\; \xi_{j}^{(p)}(t) \right) $. Then, summing each resulting equation over $j = 1,...,N$ and integrating with respect to $t$ we get
	\begin{equation}
		\label{e26}
		\displaystyle\int_0^T \left(\mathcal{L}^{\ast} \left( \partial_t \mathbf{\hat{U}}^m, \mathbf{V} \right) + \mathcal{A} \left(\mathbf{\hat{U}}^m, \mathbf{V} \right) - \mathcal{K} \left(\mathbf{\hat{U}}^m, \mathbf{V} \right)\right) \text{d}t = \displaystyle\int_0^T \mathcal{L} \left(\mathbb{S}(x,t), \mathbf{V} \right) \text{d}t.
	\end{equation}
Setting \( m = m_{\ell} \) and recalling \eqref{e24}, we pass to the weak limit and obtain
	\begin{equation}
		\label{e27}
		\displaystyle\int_0^T \left(\mathcal{L}^{\ast} \left( \partial_t \mathbf{\hat{U}}(t), \mathbf{V} \right) + \mathcal{A} \left(\mathbf{\hat{U}}(t), \mathbf{V} \right) \right) \text{d}t = \displaystyle\int_0^T  \left(\mathcal{K} \left(\mathbf{\hat{U}}(t), \mathbf{V} \right) + \mathcal{L} \left(\mathbb{S}(x,t), \mathbf{V} \right)\right) \text{d}t.
	\end{equation}
	This equality then holds for all functions $\mathbf{V} \in L^{2} \left(0,T; \mathcal{V} \right)$, as functions of the form \eqref{e25} are dense in this space (see \cite{Kouma} and \cite{J. Droniou}). Hence in particular
	\begin{equation*}
		\label{e28}
			 \mathcal{L}^{\ast} \left( \partial_t \mathbf{\hat{U}}(t), \mathbf{V} \right) + \mathcal{A} \left(\mathbf{\hat{U}}(t), \mathbf{V} \right) - \mathcal{K} \left(\mathbf{\hat{U}}(t), \mathbf{V} \right) =  \mathcal{L} \left(\mathbb{S}(x,t), \mathbf{V} \right)
	\end{equation*}
	for each $\mathbf{V} \in L^{2} \left(0,T; \mathcal{V} \right)$ and for almost every $t \in [0,T]$. Furthermore, we have $ \mathbf{\hat{U}} \in C \left( [0,T]; \mathbb{H}_2 \right)$.\\
	Now, we prove that $\mathbf{\hat{U}}$ satisifies initial conditions \eqref{e2}. By performing an integration by parts in \eqref{e27}, we first note that
	\begin{equation}
		\label{e29}
		\displaystyle\int_0^T - \mathcal{L}^{\ast} \left( \mathbf{\hat{U}}(t), \partial_t \mathbf{V} \right) + \mathcal{A} \left(\mathbf{\hat{U}}(t), \mathbf{V} \right)  \text{d}t = \displaystyle\int_0^T  \mathcal{K} \left(\mathbf{\hat{U}}(t), \mathbf{V} \right) + \mathcal{L} \left(\mathbb{S}(x,t), \mathbf{V} \right) \text{d}t + \mathcal{L} \left(  \mathbf{\hat{U}}(0), \mathbf{V}(0) \right) 
	\end{equation}
	for each $\mathbf{V} \in C^1 \left(0,T; \mathcal{V} \right)$ with $\mathbf{V}(T) = 0_{\mathbb{R}^4}$.\\
	Similarly, from \eqref{e26}, we deduce
	\begin{equation*}
		\label{e30}
		\displaystyle\int_0^T - \mathcal{L}^{\ast} \left( \mathbf{\hat{U}}^m(t), \partial_t \mathbf{V} \right) + \mathcal{A} \left(\mathbf{\hat{U}}^m(t), \mathbf{V} \right)  \text{d}t = \displaystyle\int_0^T  \mathcal{K} \left(\mathbf{\hat{U}}^m(t), \mathbf{V} \right) + \mathcal{L} \left(\mathbb{S}(x,t), \mathbf{V} \right) \text{d}t + \mathcal{L} \left(  \mathbf{\hat{U}}^m(0), \mathbf{V}(0) \right).
	\end{equation*}
	Setting $m = m_{\ell}$ and applying \eqref{e24} once more, we obtain
	\begin{equation}
		\label{e31}
		\displaystyle\int_0^T - \mathcal{L}^{\ast} \left( \mathbf{\hat{U}}(t), \partial_t \mathbf{V} \right) + \mathcal{A} \left(\mathbf{\hat{U}}(t), \mathbf{V} \right)  \text{d}t = \displaystyle\int_0^T  \mathcal{K} \left(\mathbf{\hat{U}}(t), \mathbf{V} \right) + \mathcal{L} \left(\mathbb{S}(x,t), \mathbf{V} \right) \text{d}t + \mathcal{L} \left(  \mathbf{\hat{U}}_0, \mathbf{V}(0) \right).
	\end{equation}
Since $\mathbf{V}(0)$ is arbitrary, comparing \eqref{e29} and \eqref{e31} we get $\mathbf{\hat{U}}(0) = \mathbf{\hat{U}}_0$.
\end{proof}

\subsection{Uniqueness of the weak solution}
\begin{theorem}
	\label{theoreme4}
	A weak solution of \eqref{e1}-\eqref{e2} is unique.
\end{theorem}
\begin{proof}
	 Let $T>0$ be fixed. Let us assume that there exist two weak solutions, denoted by $\mathbf{\hat{U}}$ and $\mathbf{\hat{V}}$, to problem \eqref{e1}-\eqref{e2}. We aim to prove that $\mathbf{\hat{U}} \equiv \mathbf{\hat{V}}$, i.e., $\mathbf{\hat{U}} - \mathbf{\hat{V}} \equiv 0$.\\
	 Since both $\mathbf{\hat{U}}$ and $\mathbf{\hat{V}}$ satisfy \eqref{e3}-\eqref{e4}, it follows, for all $\boldsymbol{\phi} \in \mathcal{V}$ and for almost every $t \in [0,T]$, that
	 \begin{equation}
	 	\label{e32}
	 	\mathcal{L} \left( \partial_t \mathbf{\hat{U}}(t) - \partial_t \mathbf{\hat{V}}(t), \boldsymbol{\phi} \right) + \mathcal{A} \left(\mathbf{\hat{U}}(t) - \mathbf{\hat{V}}(t), \boldsymbol{\phi} \right) - \mathcal{K} \left(\mathbf{\hat{U}}(t) - \mathbf{\hat{V}}(t), \boldsymbol{\phi} \right) = 0.
	 \end{equation}
	 Taking $\boldsymbol{\phi} = \mathbf{\hat{U}} - \mathbf{\hat{V}}$ in \eqref{e32} yields for almost every $t \in [0,T]$:
	 \begin{align*}
	 	\label{e33}
	 	\begin{split}
	 		& \dfrac{\text{d}}{\text{d}t} \left( \dfrac{1}{2} \left\Vert \mathbf{\hat{U}}(t) - \mathbf{\hat{V}}(t) \right\Vert_{\mathbb{H}_2}^2 \right) + \mathcal{A} \left(\mathbf{\hat{U}}(t) - \mathbf{\hat{V}}(t), \mathbf{\hat{U}}(t) - \mathbf{\hat{V}}(t) \right) - \mathcal{K} \left(\mathbf{\hat{U}}(t) - \mathbf{\hat{V}}(t), \mathbf{\hat{U}}(t) - \mathbf{\hat{V}}(t) \right) \\ & = \mathcal{L}^{\ast} \left( \partial_t \mathbf{\hat{U}}(t) - \partial_t \mathbf{\hat{V}}(t), \mathbf{\hat{U}}(t) - \mathbf{\hat{V}}(t) \right) + \mathcal{A} \left(\mathbf{\hat{U}}(t) - \mathbf{\hat{V}}(t), \mathbf{\hat{U}}(t) - \mathbf{\hat{V}}(t) \right) - \mathcal{K} \left(\mathbf{\hat{U}}(t) - \mathbf{\hat{V}}(t), \mathbf{\hat{U}}(t) - \mathbf{\hat{V}}(t) \right) =0.
	 	\end{split}
	 \end{align*}
	 Using inequalities \eqref{e14}-\eqref{e16}, we obtain
	 \begin{equation*}
	 	\label{e34}
	 	\dfrac{\text{d}}{\text{d}t} \left\Vert \mathbf{\hat{U}}(t) - \mathbf{\hat{V}}(t) \right\Vert_{\mathbb{H}_2}^2 \leq \kappa \left\Vert \mathbf{\hat{U}}(t) - \mathbf{\hat{V}}(t) \right\Vert_{\mathbb{H}_2}^2
	 \end{equation*}
	 for a certain constant $\kappa >0$ and for almost every $t \in [0,T]$. Gronwall's inequality imply
	 \begin{equation*}
	 	\left\Vert \mathbf{\hat{U}}(t) - \mathbf{\hat{V}}(t) \right\Vert_{\mathbb{H}_2}^2 \leq \mathrm{e}^{\kappa t } \left\Vert \mathbf{\hat{U}}(0) - \mathbf{\hat{V}}(0) \right\Vert_{\mathbb{H}_2}^2 = 0 
 	 \end{equation*} 
 	 since $\mathbf{\hat{U}}$ and $\mathbf{\hat{V}}$ satisfy \eqref{e4}.\\
 	 Consequently, $\mathbf{\hat{U}} \equiv \mathbf{\hat{V}}$.
 	 
\end{proof}

\section{Additional regularity results}
The regularity of the solution increases with that of the data. More precisely, we have the following result.
\begin{theorem}[\textbf{Gain of regularity}]
	\label{theoreme5}
	Let $ T > 0 $ be fixed. Assume that $ \mathbf{\hat{U}}_0 \in \mathcal{V} $, $ \mathbb{S} \in L^2(0,T; \mathbb{H}_2) $ and 
	\begin{equation*}
	\mathbf{\hat{U}} \in L^2 \left( 0,T; \mathcal{V} \right) \quad \text{with} \quad \partial_t \mathbf{\hat{U}} \in L^2(0,T; \mathcal{V}')
	\end{equation*}
	is the weak solution to the system \eqref{e1}-\eqref{e2}. Then,
	\begin{equation*}
	\mathbf{\hat{U}} \in L^2 \left(0,T; (H^2(0,1))^4 \right) \cap L^\infty(0,T; \mathcal{V}), \quad \partial_t \mathbf{\hat{U}} \in L^2(0,T; \mathbb{H}_2),
	\end{equation*}
	and
	\begin{equation*}
		\label{e35}
\underset{0 \leq t \leq T}{\mathrm{ess} \; \mathrm{sup}} \left( \left\Vert \mathbf{\hat{U}}(t) \right\Vert_{\mathcal{V}} + \left\Vert \mathbf{\hat{U}} \right\Vert_{L^2 \left( 0,T; \left( H^2(0,1) \right)^4 \right)} + \left\Vert \partial_t \mathbf{\hat{U}} \right\Vert_{L^2(0,T;\mathbb{H}_2)} \right) \leq C \left( \left\Vert \mathbb{S} \right\Vert_{L^2(0,T;\mathbb{H}_2)} + \left\Vert \mathbf{\hat{U}}_0 \right\Vert_{\mathcal{V}} \right),
	\end{equation*}
	for some constant $ C > 0 $.
\end{theorem}
\begin{proof}
	Let $m \in \mathbb{N}^{\ast}$ and $T>0$ be fixed.\\ For $j=1,...,m$ and $t \in [0,T]$, we multiply \eqref{e9} by $ \text{diag } \left(\dfrac{\text{d}f_j^m}{\text{d}t}(t),\; \dfrac{\text{d}s_j^m}{\text{d}t}(t),\; \dfrac{\text{d}g_j^m}{\text{d}t}(t),\; \dfrac{\text{d}p_j^m}{\text{d}t}(t) \right) $. Then, summing each resulting equation over $j = 1,...,m$, we have
	\begin{equation*}
		\mathcal{L} \left( \partial_t \mathbf{\hat{U}}^m(t),  \partial_t \mathbf{\hat{U}}^m(t) \right) + \mathcal{A} \left(\mathbf{\hat{U}}^m(t), \partial_t \mathbf{\hat{U}}^m(t) \right) = \mathcal{K} \left(\mathbf{\hat{U}}^m(t), \partial_t \mathbf{\hat{U}}^m(t) \right) + \mathcal{L} \left(\mathbb{S}(x,t), \partial_t \mathbf{\hat{U}}^m(t) \right)
		\end{equation*}
		for almost every $ t \in [0,T] $. We write
		\begin{equation*}
			\mathcal{A} \left(\mathbf{\hat{U}}^m(t), \partial_t \mathbf{\hat{U}}^m(t) \right) = \mathcal{B} \left(\mathbf{\hat{U}}^m(t), \partial_t \mathbf{\hat{U}}^m(t) \right) + \mathcal{C} \left(\mathbf{\hat{U}}^m(t), \partial_t \mathbf{\hat{U}}^m(t) \right),
		\end{equation*}
		where
		\begin{equation*}
			\mathcal{B} \left(\mathbf{\hat{U}}^m(t), \partial_t \mathbf{\hat{U}}^m(t) \right) = \begin{bmatrix}
				 \left\langle E_f \partial_x \hat{U}_f^m, \partial_{xt}^2 \hat{U}_f^m \right\rangle\\
				 \left\langle E_s \partial_x \hat{U}_s^m, \partial_{xt}^2 \hat{U}_s^m \right\rangle\\
				\left\langle E_g \partial_x \hat{U}_g^m, \partial_{xt}^2 \hat{U}_g^m \right\rangle\\
				\left\langle E_p \partial_x \hat{U}_p^m, \partial_{xt}^2 \hat{U}_p^m \right\rangle
			\end{bmatrix} \text{ and } \mathcal{C} \left(\mathbf{\hat{U}}^m(t), \partial_t \mathbf{\hat{U}}^m(t) \right) = \mathcal{A} \left(\mathbf{\hat{U}}^m(t), \partial_t \mathbf{\hat{U}}^m(t) \right) - \mathcal{B} \left(\mathbf{\hat{U}}^m(t), \partial_t \mathbf{\hat{U}}^m(t) \right).
		\end{equation*}
		It is easy to see that
		\begin{equation*}
			\mathcal{B} \left(\mathbf{\hat{U}}^m(t), \partial_t \mathbf{\hat{U}}^m(t) \right) = \dfrac{1}{2} \dfrac{\text{d}}{\text{d}t}  \mathcal{B} \left(\mathbf{\hat{U}}^m(t), \mathbf{\hat{U}}^m(t) \right).
		\end{equation*}
		Furthermore, let $\varepsilon >0$ and apply Young's inequality $ \left( ab \leq \varepsilon a^2 + \dfrac{b^2}{4 \varepsilon} \right) $ (see \cite{L.C. Evans}) to obtain
		\begin{align}
			\label{e36}
		\begin{split}
				\mathcal{C} \left(\mathbf{\hat{U}}^m(t), \partial_t \mathbf{\hat{U}}^m(t) \right) & \preceq \dfrac{C}{\varepsilon} \mathbb{I}_4 \left( \left\Vert \hat{U}_{\alpha}^m \right\Vert_{\mathcal{V}_{\alpha}}^2 \right)_{\alpha \in \mathcal{I}}^{\top} + \varepsilon \mathbb{I}_4 \left( \left\Vert \partial_t \hat{U}_{\alpha}^m \right\Vert_{2}^2 \right)_{\alpha \in \mathcal{I}}^{\top} \\ 
			 \mathcal{L} \left(\mathbb{S}(x,t), \partial_t \mathbf{\hat{U}}^m(t) \right) & \preceq \dfrac{C}{\varepsilon} \mathbb{I}_4  \left(  \left\Vert \mathbb{S}_i(x,t) \right\Vert_{2}^2  \right)_{i=1,...,4}^{\top}  + \varepsilon \mathbb{I}_4 \left( \left\Vert \partial_t \hat{U}_{\alpha}^m \right\Vert_{2}^2 \right)_{\alpha \in \mathcal{I}}^{\top},
		\end{split}
		\end{align}
		where $ \mathbb{I}_4 $ stands for the identity matrix in $ \mathbb{R}^{4 \times 4} $. By adding inequalities \eqref{e36} term by term, we find
		\begin{equation*}
			\left( \left\Vert \partial_t \hat{U}_{\alpha}^m \right\Vert_{2}^2 \right)_{\alpha \in \mathcal{I}}^{\top} + \dfrac{1}{2} \dfrac{\text{d}}{\text{d}t}  \mathcal{B} \left(\mathbf{\hat{U}}^m(t), \mathbf{\hat{U}}^m(t) \right) \preceq \dfrac{C}{\varepsilon} \mathbb{I}_4 \left( \left( \left\Vert \hat{U}_{\alpha}^m \right\Vert_{\mathcal{V}_{\alpha}}^2 \right)_{\alpha \in \mathcal{I}}^{\top} + \left(  \left\Vert \mathbb{S}_i(x,t) \right\Vert_{2}^2  \right)_{i=1,...,4}^{\top} \right) + 2 \varepsilon \left( \left\Vert \partial_t \hat{U}_{\alpha}^m \right\Vert_{2}^2 \right)_{\alpha \in \mathcal{I}}^{\top}.
		\end{equation*}
	 Summing over each components for $\varepsilon = \dfrac{1}{4}$ and integrating over $[0,T]$, we find
		\begin{equation}
			\label{e37}
		\sup_{0\leq t \leq T}	 \left\Vert \mathbf{\hat{U}}^m(t) \right\Vert_{\mathcal{V}}^2   \leq C \left( \left\Vert \mathbb{S} \right\Vert_{L^2(0,T;\mathbb{H}_2)}^2 + \left\Vert \mathbf{\hat{U}}_0 \right\Vert_{\mathcal{V}}^2 \right)
		\end{equation}
		according to Theorem \ref{theoreme2}, where we estimated $\left\Vert \mathbf{\hat{U}}^m(0) \right\Vert_{\mathcal{V}} \leq \left\Vert \mathbf{\hat{U}}_0 \right\Vert_{\mathcal{V}}$. Passing to limits as $m = m_{\ell} \to \infty$, we deduce $\mathbf{\hat{U}} \in L^{\infty} \left(0,T;\mathcal{V}\right)$, $\partial_t \mathbf{\hat{U}} \in L^2 \left(0,T; \mathbb{H}_2\right)$.\\
		It is clear that for all $\boldsymbol{\phi} \in \mathcal{V}$, we have
		\begin{equation*}
			\mathcal{L} \left( \partial_t \mathbf{\hat{U}}^m(t),  \boldsymbol{\phi} \right) + \mathcal{A} \left(\mathbf{\hat{U}}^m(t),\boldsymbol{\phi} \right) = \mathcal{K} \left(\mathbf{\hat{U}}^m(t), \boldsymbol{\phi} \right) + \mathcal{L} \left(\mathbb{S}(x,t), \boldsymbol{\phi} \right)
		\end{equation*}
		for almost every $t \in [0,T]$. It follows that
		\begin{equation*}
				 \mathcal{A} \left(\mathbf{\hat{U}}^m(t), \boldsymbol{\phi} \right) = \mathcal{K} \left(\mathbf{\hat{U}}^m(t), \boldsymbol{\phi} \right) + \mathcal{L} \left(\mathbb{S}(x,t) - \partial_t \mathbf{\hat{U}}^m, \boldsymbol{\phi} \right)
		\end{equation*}
		for almost every $t \in [0,T]$. Since the norms on each space $\mathcal{V}_\alpha, \; \alpha \in \mathcal{I}$, are equivalent to the standard $H_0^1(0,1)$-norm, and since $\mathbb{S}(x,t) - \partial_t \mathbf{\hat{U}}^m(t) \in \mathbb{H}_2$ for almost every $t \in [0,T]$, we deduce, from Theorem 4 in \cite{L.C. Evans} (pp. 317), that $\mathbf{\hat{U}}(t) \in \left(H^2(0,1)\right)^4$ almost everywhere in $[0,T]$, with the estimate
		\begin{equation}
			\label{e38}
			\left\Vert \mathbf{\hat{U}(t)} \right\Vert_{\left(H^2(0,1)\right)^4}^2 \leq C \left(  \left\Vert \mathbb{S}(x,t) \right\Vert_{\mathbb{H}_2}^2 + \left\Vert \partial_t \mathbf{\hat{U}} \right\Vert_{\mathbb{H}_2}^2 +  \left\Vert \mathbf{\hat{U}} \right\Vert_{\mathbb{H}_2}^2 \right).
		\end{equation}
	Integrating over $[0,T]$ and using the estimates given in \eqref{e37}, we complete the proof.
	
\end{proof}

\section*{Conclusion}

This work presents a detailed mathematical analysis of the coupled system $(\Sigma_T)$, which models multi-phase thermal exchanges with strong, asymmetric interfacial interactions. The main analytical challenge was due to the complex structure of the coupling terms, particularly the non-symmetric and nonlinear cross-diffusion effects. To address these, we developed a Faedo-Galerkin approximation scheme, combined with suitable basis transformations to simplify certain interactions. This approach allowed us to demonstrate the existence, uniqueness, and regularity of weak solutions in a nonlinear and strongly coupled parabolic setting. Our findings broaden the classical theory of reaction-convection-diffusion systems by covering a new class of multilayer energy exchange models, beyond the usual symmetry or weak coupling assumptions. In addition to its theoretical interest, this study provides a solid foundation for the development of reliable numerical simulations in the context of advanced thermal technologies.

\end{document}